\documentclass{article}
\usepackage[utf8]{inputenc}
\usepackage[bookmarksnumbered=true]{hyperref} 
\hypersetup{
     colorlinks = true,
     linkcolor = blue,
     anchorcolor = blue,
     citecolor = teal,
     filecolor = blue,
     urlcolor = blue
     }
\title{Lattice Conditional Independence Models and Hibi Ideals}

\author{
Peter Caines, Fatemeh Mohammadi, Eduardo Sáenz-de-Cabezón, and Henry Wynn
 }

\usepackage{graphicx}
\usepackage{amsmath}

\usepackage{amscd,amsmath,amssymb}
\usepackage{amsthm}
\usepackage[all]{xy}
\usepackage{color}
\usepackage{hyperref}
\usepackage{tikz, float} 
\usetikzlibrary {positioning}

\def\frk{\mathfrak}               

\def\Phi{{\frk n}}
\def\Phi{{\frk N}}

\def\A{{\mathcal A}}

\def\ji{J({\mathcal L})}
\def\ml{{\mathcal L}}
\def\opn#1#2{\def#1{\operatorname{#2}}} 
\opn\chara{char} \opn\length{\ell} \opn\pd{pd} \opn\rk{rk}
\opn\projdim{proj\,dim} \opn\injdim{inj\,dim} \opn\rank{rank}
\opn\depth{depth} \opn\grade{grade} \opn\height{height}
\opn\embdim{emb\,dim} \opn\codim{codim}

\opn\Tr{Tr} \opn\bigrank{big\,rank}
\opn\superheight{superheight}\opn\lcm{lcm}
\opn\trdeg{tr\,deg}
\opn\reg{reg} \opn\lreg{lreg} \opn\ini{in} \opn\lpd{lpd}
\opn\size{size} \opn\sdepth{sdepth}
\opn\link{link}\opn\fdepth{fdepth}\opn\lex{lex}
\opn\LM{LM}
\opn\LC{LC}
\opn\NF{NF}
\opn\Merge{Merge}
\opn\sgn{sgn}
\opn\div{div} \opn\Div{Div} \opn\cl{cl} \opn\Pic{Pic}
\opn\Prin{Prin}
\opn\op{op}
\opn\indeg{indeg} \opn\outdeg{outdeg}
\opn\red{red}
\opn\Spec{Spec} \opn\Supp{Supp} \opn\supp{supp} \opn\Sing{Sing}
\opn\Ass{Ass} \opn\Min{Min}\opn\Mon{Mon} \opn\val{val}
\opn\Ann{Ann} \opn\Rad{Rad} \opn\Soc{Soc}
 \opn\Ker{Ker} \opn\Coker{Coker} \opn\Am{Am}
\opn\Hom{Hom} \opn\Tor{Tor} \opn\Ext{Ext} \opn\End{End}
\opn\Aut{Aut} \opn\id{id}

\opn\nat{nat}
\opn\pff{pf}
\opn\Pf{Pf} \opn\GL{GL} \opn\SL{SL} \opn\mod{mod} \opn\ord{ord}
\opn\Gin{Gin} \opn\Hilb{Hilb}\opn\sort{sort}
\opn\span{span}
\opn\Image{Image}

\opn\aff{aff} \opn\con{conv} \opn\relint{relint} \opn\st{st}
\opn\lk{lk} \opn\cn{cn} \opn\core{core} \opn\vol{vol}
\opn\link{link} \opn\star{star}\opn\lex{lex}\opn\set{set}
\opn\dist{dist}
\opn\gr{gr}

\def\pot#1#2{#1[\kern-0.28ex[#2]\kern-0.28ex]}

\opn\dirlim{\underrightarrow{\lim}}
\opn\inivlim{\underleftarrow{\lim}}
\def\Implies{\ifmmode\Longrightarrow \else
        \unskip${}\Longrightarrow{}$\ignorespaces\fi}
\def\implies{\ifmmode\Rightarrow \else
        \unskip${}\Rightarrow{}$\ignorespaces\fi}
\def\iff{\ifmmode\Longleftrightarrow \else
        \unskip${}\Longleftrightarrow{}$\ignorespaces\fi}

\let\:=\colon

\definecolor{caribbeangreen}{rgb}{0.0, 0.8, 0.6}
\definecolor{capri}{rgb}{0.0, 0.75, 1.0}

\newtheorem{Theorem}{Theorem}[section]
\newtheorem{Lemma}[Theorem]{Lemma}
\newtheorem{Corollary}[Theorem]{Corollary}

\theoremstyle{definition}
\newtheorem{Example}[Theorem]{Example}
\newtheorem{Remark}[Theorem]{Remark}

\newtheorem{Definition}[Theorem]{Definition}

\let\kappa=\varkappa
\def\qed{\ifhmode\textqed\fi
      \ifmmode\ifinner\quad\qedsymbol\else\dispqed\fi\fi}
\def\textqed{\unskip\nobreak\penalty50
       \hskip2em\hbox{}\nobreak\hfil\qedsymbol
       \parfillskip=0pt \finalhyphendemerits=0}
\def\dispqed{\rlap{\qquad\qedsymbol}}

\opn\dis{dis}
\def\pnt{{\raise0.5mm\hbox{\large\bf.}}}

\opn\Lex{Lex}
\opn\syz{{\rm syz}}
\opn\spoly{{\rm spoly}}
\opn\LM{{\rm LM}}
\opn\lm{{\rm lm}}
\opn\lcm{{\rm lcm}} \opn\A{\mathcal A}

\numberwithin{equation}{section}

\newcommand{\bigCI}{\mathrel{\text{\scalebox{1.07}{$\perp\mkern-10mu\perp$}}}}

\begin{document}

\maketitle



\tikzstyle{Cgray}=[scale = .45,circle, fill = gray, minimum size=3mm] \tikzstyle{Cblack}=[scale = .7,circle, fill = black, minimum size=3mm]
\tikzstyle{Cblue}=[scale = .5,circle, fill = blue, inner sep = 0pt, minimum
size=3mm]
\tikzstyle{C1}=[scale = .7,circle, fill = black!0, inner sep = 0pt, minimum
size=3mm]

\tikzstyle{test2}=[scale = 1.5,circle, fill = black!0, inner sep = 0pt, minimum
size=3mm]
\tikzstyle{Cwhite}=[scale = .45,circle, fill = white, minimum size=3mm] 
\tikzstyle{Cblack2}=[scale = .3,circle, fill = black, minimum size=3mm] 
\tikzstyle{Cblack}=[scale = .3,circle, fill = black, minimum size=3mm]
\tikzstyle{C0}=[scale = .9,circle, fill = black!0, inner sep = 0pt, minimum size=3mm]
\tikzstyle{C1}=[scale = .7,circle, fill = black!0, inner sep = 0pt, minimum size=3mm]
\tikzstyle{Cred}=[scale = .4,circle, fill = red, minimum size=3mm] 
\tikzstyle{Cblack3}=[scale = .4,circle, fill = black, minimum size=3mm]

\maketitle

\noindent{\bf Abstract.}
Lattice Conditional Independence models \cite{andersson1993lattice} 
are a class of models developed first for the Gaussian
case in which a distributive lattice classifies all the conditional independence statements. The main result is that these models can equivalently be described via a transitive directed acyclic graph (TDAG) in which, as is normal for causal models, the conditional independence is in terms of conditioning on ancestors in the graph. We demonstrate that a parallel stream of research in algebra, the theory of Hibi ideals, not only maps directly to the LCI models but gives a vehicle to generalise the theory from the linear Gaussian case. Given a distributive lattice (i) each conditional independence statement is associated with a Hibi relation defined on the lattice, (ii) the directed graph is given by chains in the lattice which correspond to chains of conditional independence, (iii) the elimination ideal of product terms in the chains gives the Hibi ideal and (iv) the TDAG can be recovered from a special bipartite graph constructed via the Alexander dual of the Hibi ideal. It is briefly demonstrated that there are natural applications to statistical log-linear models, time series, and Shannon information flow.

\section{Introduction}
Lattice Conditional Independence (LCI) are a special type of statistical graphical models introduced by Andersson and Perlman
\cite{andersson1993lattice} in the context of linear Gaussian models. The idea is that in the case of a distributive lattice of linear subspaces associated with the marginal models, all Conditional Independence (CI) statements could be classified via the intersections (meets) on the lattice. The theory comes with an important duality which links the LCI models isomorphically to a Transitive Directed Acyclic Graph (TDAG) from \cite{andersson1995relation}, so that LCI could be seen as providing an important class of causal models. More precisely, each intersection on the lattice is mapped to a conditional independence statement conditioning on common ancestors in the TDAG.

The LCI models that arise from a TDAG are not in general graphical models based on directed acyclic graphs (DAGs) or undirected graphs (UGs). An important subclass of a graphical model is a decomposable graphical model, see~\cite{lauritzen1996graphical}, and it is possible to confuse these with the LCI models. There is an overlap between the two classes, see~\cite{andersson1995relation, castelo2003characterization}. An added complexity is that some DAG models give the same set of conditional independences as UGs; The DAGs equivalent to the same model are called {\em Markov equivalent}.  The TDAG models have a unique set of conditional independence statements but are not in general represented by UGs.

In this paper, we study LCI models through the lens of algebraic statistics, by connecting them to so-called Hibi rings that are both predicated on distributive (Boolean) lattices.
 Hibi rings and their properties have been extensively studied in commutative algebra, see e.g.~\cite{Hibi1, Herzog, ene2011monomial}. The key point linking the two areas is the fact that for LCI models every member of the lattice is associated with a marginal distribution. This means that, initially, the algebraic relations are associated with margins. The existence of a connection between LCI models and Hibi ideals was suggested by Beerenwinkel, Eriksson and Sturmfels \cite{beerenwinkel2007conjunctive} within the context of Gr\"obner basis theory and maximum likelihood estimation. Here, we develop the algebraic connection in more detail, particularly in the relationship to the duality with the TDAG representation. The insight is that the TDAG representation that is related to conditional probability statements, suggests that there should be a representation of this duality in the algebraic context. This is indeed the case and, remarkably, the relation is that the Hibi (marginal) representation is the toric elimination ideal from the products of conditional probabilities given by the TDAG. Moreover, the TDAG itself can be read from a special bipartite graph arising from a particular Alexander duality.  

In \cite{caines2007algebraic} it was suggested that LCI models have wide applicability, and in that paper, the application to time series was discussed. The work by \cite{jozsa2019causality} is a continuation of this approach in which a more formal link to Granger causality has been made. This, in particular, prompts the inclusion of time series as one of the main applications. Another application is to Shannon information where the LCI models lead to additive models for information.

The outline of the paper is as follows. 
In \S\ref{Sec:Elem}, we will give an elementary review of the LCI models with their connections to other types of statistical models, and in \S\ref{Sec:LCI}, we discuss the linear (Gaussian) case. In Section~\ref{Sec:Hibi}, we will describe the Hibi ideals. Section~\ref{Sec:TDAG} will link all the concepts together. This includes an understanding of the effect of reversing the arrows in the TDAG. The final sections cover the two applications: time series and Shannon information.

\subsection{Elementary considerations}\label{Sec:Elem}
In teaching elementary statistics and probability theory, there is something of a conundrum in the expression of independence and conditional independence. This is seen in as simple an example as two independent binary random variables $X_1,X_2 = \{0,1\}$. 
Let 
$$p_{ij} = \mbox{prob} \{(X_1 = i) \cap (X_2 = j)\}\;(i,j=0,1). $$
The issue is that there are two ways of expressing independence. Writing the margins as
$$ p_{i +} = \sum_{j=0,1} p_{ij} \; \text{ and } \; p_{+ j} = \sum_{i=0,1} p_{ij},$$
the first way of expressing independence is 
$$p_{ij} = p_{i +} p_{ + j}\; \; (i,j = 0,1).$$
The second, and equivalent, condition (given all the $p_{ij} > 0$) is  
$$p_{00}p_{11} - p_{10} p_{01}=0.$$
In the algebraic context, this can be seen as a toric ideal condition in a certain polynomial ring, which is at the foundation of algebraic statistics for log-linear models. That is to say
it arises as the elimination ideal from the power product representation
$$p_{ij} = s_i t_j,$$
or its log-linear equivalent
$$ \log p_{ij} = \theta_i + \phi_j.$$
where $s_i = \exp(\theta_i)$ and  $t_j = \exp(\phi_j).$

\smallskip

With three binary random variables $X_1,X_2,X_3$, for the conditional independence statement
$$ X_1  \bigCI X_2 | X_3,$$
we have either the marginal version
\begin{equation}
p_{ijk} p_{++ k}  = p_{i + k} p_{+ j k}\;\; (i,j,k = 0,1) \label{margin},
\end{equation}
or the toric ideal with one generator for each level of $X_3$
\begin{eqnarray*}
p_{000}p_{110} - p_{010} p_{100} & = & 0 \\
p_{001}p_{111} - p_{011} p_{101} & = & 0. 
\end{eqnarray*}
The key idea in this paper is that the version \eqref{margin} relies on a special set of margins indexed by subsets of indices $\{ijk,jk,ik,k\}$, which form a lattice generated by the $2$-subsets $\{jk,ik\}$ 
under unions and intersections. 
We note, immediately, that  there is a version of \eqref{margin}  for discrete random variables or continuous random variables and for larger collections of index sets. Let $[n] = \{1, \ldots, n\}$. For any subset of indices $ I \subset [n]$, then, with slight abuse of notation we write $p_I$ for the $I$-margin (summing or integrating over variables in $[n] \setminus I$).

\medskip

Thus, let $I, J \subset [n]$. Then the conditional independence statement
\begin{equation}
X_{I \setminus J} \bigCI X_{J \setminus I} | X_{I \cap J},
 \label{CI}
\end{equation}is represented by
\begin{equation}
p_{I \cup J} p_{I \cap J} = p_I p_J\quad\text{(given that $p_I>0$ for all $I$).} \label{P_I}
\end{equation}
This is somewhat lazy notation: we should write this strictly as:
\begin{equation}
p_{I \cup J}(x_{I \cup J}) p_{I \cap J}(x_{I \cap J}) = p_I (x_I) p_J(x_J) \label{sets}, 
\end{equation}
for all $ x = (x_1, \ldots, x_n)$ in the support of the distribution.

This shows the relationship with the sublattice generated by $I$ and $J$ in which the join operator is the union of sets $\cup$, and the meet operator is the intersection of sets $\cap$.
More precisely, 
a term like $p_I$ is used both as terms in lattice theory and a (possibly marginal) probability distribution, in the continuous case a probability density function. We see clearly also how from an abstract point of view, 
\eqref{P_I} is a toric term indexed by the lattice. 


One aim of the present paper is to show that if we have a Boolean (distributive) lattice formed from collection of sets $\mathbb I$, then there is a probability model associated with $\mathbb I$ such that {\em every} union/intersection formula of the type \eqref{P_I} gives a conditional independence statement of the type \eqref{sets}.  Moreover, we demonstrate the connection with the associated toric ideal of the lattice.


The toric ideal associated with a distributive lattice is called a Hibi ideal after the work of Hibi in \cite{Hibi1}.  Moreover, it does indeed hold that this defines a class of probability models which is legitimate to call Lattice Conditional Independence (LCI)
because exactly such models were introduced in the Gaussian (normal) case. 
The advantages of the algebraic approach using Hibi ideals are that, (i) we can define LCI models for any type of random variables, discrete or continuous, (ii) entities such as the  Transitive Directed Acyclic Graph (TDAG) associated with the LCI model have a purely algebraic realisation and (iii) the ideal theory can be used  to make statements about the probability models. One algebraic concept is that of the  Alexander duality, which we will show has a precise relationship to the underlying TDAG.

\begin{Definition}[TDAG]\label{def:TDAG}
A directed acyclic graph (DAG) $G=([n],E)$ is transitive (TDAG) if $\{i\rightarrow j\}\in E$ and $\{j\rightarrow k\}\in E$ imply that $\{i\rightarrow k\}\in E$ (see Figures~\ref{fig:TDAG}  and~\ref{fig:time}).
\end{Definition}

\subsection{Linear LCI models}\label{Sec:LCI}
Multivariate normal LCI models are defined in terms of collections of marginal distributions based on a distributive lattice of subsets of indices, as described above and will be made more formal later in the paper. The meet and join of the lattice will give intersections and unions of vector spaces associated with the margins.

Let us now discuss a single conditional independence statement as in \eqref{CI} for a (zero mean) multivariate normal random vector $X= (X_1, \ldots, X_n)^T$. Let the full rank covariance matrix be $\Gamma$. 
First, we factor $\Gamma$ in some standard way: $\Gamma = A A^T$, where $A$ is a square matrix.  We can then think of $X$ as arising from what is sometimes called a moving average representation
$$X = A Z,$$
where $Z$ is the zero mean multivariate normal random variable with the covariance 
matrix $I_n$ (the $n \times n$ identity matrix). For index  sets $I$, $J$ and $K = I \cap J$ we can express, by selecting appropriate rows of $A$, 
the corresponding random vector as 
$$X_I =  A_I Z,\; X_J = A_J Z,\; X_K = A_K Z,$$
with covariance matrices, respectively:
$$\Gamma_I = A_I A_I^T,\; \Gamma_J = A_J A_J^T,\; \Gamma_K = A_K A_K^T.$$
(Note that $A_I$ denotes the submatrix of $A$ on the rows indexed by $I$). Using a standard formula known as Schur complement, the covariance matrix between $X_I$ and $X_J$ conditional on $X_K$ is
\begin{equation}
A_I A_J^T - A_I A_K^T (A_K A_K^T)^{-1}A_K A_J^T. \label{cond}
\end{equation}
For conditional independence in the Gaussian case, it is necessary and sufficient that the matrix at \eqref{cond}  is zero.

Notice that $P_K = A_K^T (A_K A_K^T)^{-1}A_K$ is the orthogonal projection operator onto the row space of $A_K$.
Similarly, we can construct the orthogonal projections $P_I = A_I^T (A_I A_I^T)^{-1}A_I$ and $P_J = A_J^T (A_J A_J^T)^{-1}A_J$ for the row spaces of $A_I$ and $A_J$, respectively. The conditional independence condition can then be seen to be equivalent to
$$P_I (I_n -P_K) P_J = 0.$$
It is important to note that the  condition is equivalent to the commutativity condition
$$P_I P_J = P_J P_I\;\; ( = P_K) ,$$
which is necessary and sufficient for the conditional independence.

If we extend the argument to all $X_I, I \in \mathbb I$, where $\mathbb I$ is a distributive sublattice, then we see that it is necessary and sufficient that all projections $P_I,$ for all $I \in \mathbb I$ in the lattice commute. This implies that they are simultaneously diagonalisable with spectrum (eigenvalues) being 0 or 1.  In fact the
location of the values 1 in $[n]$ in the spectrum of each projector $I,J,K \in [n]$ corresponds precisely to the indicator of the corresponding index set:  
$P_I = \mbox{diag}(q_1, \dots, q_n)$ where $q_i = 1$ for $i \in I $ and 0 otherwise. For example,  for the sets $I = 13, J = 23$, and $K = 3$ we have the eigenvalues (spectrum)
$\{1, 0, 1\}  \{0,1,1\}$ and $\{0,0,1\}$.
The sets where the spectrum is non-zero
are a model for a subset lattice which is equivalent to the original lattice.   

Important in the theory of LCI is that every distributive lattice, which we here associate with the collection of index sets $\mathbb I$, arises from a TDAG and vice versa:  any distributive lattice defines a TDAG. Although we will study this in more detail in \S\ref{sec:TDAG_L} (see also Definition~\ref{def:TDA_LG}), we can simply describe this 1-1 correspondence as follows. The corresponding TDAG, $G([n],E)$ has $n$ vertices labelled by $1,\ldots,n$, where each $i$ corresponds to an $X_i, \;i=1,\ldots, n$ in our probability model, and $i \rightarrow j$ is in the edge set $E$ if and only if $i$ is in the ancestral set of $j$, meaning:  $i \leq j$ in the partial ordering defined by the lattice. 
  
\section{Hibi ideals and LCI models}\label{Sec:Hibi}
In \cite{Hibi1}, Hibi introduced a class of algebras which nowadays  are called Hibi rings. They are  toric rings attached to finite posets, and may be viewed as natural generalizations of polynomial rings. We first recall the construction of Hibi ideals associated to distributive lattices. Recall that a poset in which every two elements have a ``meet" and a ``join" is called a lattice.
A lattice 
$(\mathcal{L},\vee,\wedge)$ is {\em distributive} if the following additional identities hold for all $x, y$, and $z$ in $\mathcal{L}$:
\[
x\wedge (y\vee z)=(x\wedge y)\vee (x\wedge z)\text{ and } x\wedge (y\vee z)=(x\vee y)\wedge (x\vee z).
\]
Given a poset $Q$, a non-empty subset $I\subseteq Q$ is called an {\it order ideal} of $Q$ if for any $i\in I$ and $j\in Q$, $j<i$ implies that $j\in I$.
We denote $\mathcal{O}(Q)$ for the set of order ideals of $Q$.
Note that  
if we define the union of sets $\cup$ as the join operator and the intersection of sets $\cap$ as the meet operator, then $(\mathcal{O}(Q),\cap,\cup)$ is naturally a distributive lattice.

Before defining Hibi ideals, we first fix our notation. Consider the poset $Q$ on the elements $1,2,\ldots,n$ and the polynomial ring $R=\mathbb{C}[z_1,\ldots,z_n, y_1,\ldots,y_n]$ over the field of complex numbers on the variables $z_i,y_i$, where variables $z_i,y_i$ are corresponding to the element $i$ of $Q$. For every order ideal $I\subseteq Q$ we associate the monomials $u_I:=\prod_{i\in {I}}z_i\prod_{i\not\in {I}}y_i$ and $u'_I:=\prod_{i\in {I}}z_i$. 
Then we define the following polynomial maps:
\begin{eqnarray}\label{eq:1}
\varphi_Q:\ \mathbb C[p_{I}: I\in \mathcal{O}(Q)]\rightarrow \mathbb{C}[z_1,\ldots,z_n, y_1,\ldots,y_n]\quad {\rm with}\quad p_I\mapsto u_I,
\end{eqnarray}
and 
\begin{eqnarray}\label{eq:2}
\varphi'_Q:\ \mathbb C[p_{I}: I\in \mathcal{O}(Q)]\rightarrow \mathbb{C}[z_1,\ldots,z_n]\quad {\rm with}\quad p_I\mapsto u'_I.
\end{eqnarray}
\begin{Definition}[Hibi ideal]
The {\em Hibi ideal} of $Q$ is defined as $L_Q={\rm Ker}(\varphi_Q)={\rm Ker}(\varphi'_Q)$.
\end{Definition}
The equality above 
is implicit in Hibi's seminal paper \cite{Hibi1}, see also~\cite[Theorem~1.1]{ene2011monomial}.

\begin{figure}
 \begin{center}
\begin{tikzpicture} [scale = .6]


  \node (0) at (15,1) {$y_1y_2y_3y_4y_5$};

  \node (3) at (15,3) {$z_3y_1y_2y_4y_5$};
  \node (34) at (17,5) {$z_3z_{4}y_1y_2y_5$};
  \node (23) at (13,5) {$z_2z_3y_1y_4y_5$};
  
  \node (234) at (15,7) {$z_2z_3z_4y_1y_5$};
  \node (123) at (11,7) {$z_1z_2z_3y_4y_5$};
  \node (345) at (19,7) {$z_3z_4z_5y_1y_2$};

  \node (1234) at (13,9) {$z_1z_2z_3z_4y_5$};
  \node (2345) at (17,9) {$z_2z_3z_4z_5y_1$};

  \node (12345) at (15,11) {$z_1z_2z_3z_4z_5$};

\draw[black](0)--(3);
\draw[black](3)--(34);
\draw[black](3)--(23);
\draw[black](23)--(234);
\draw[black](34)--(234);
\draw[black](34)--(345);
\draw[black](23)--(123);
\draw[black](123)--(1234);
\draw[black](234)--(1234);
\draw[black](234)--(2345);
\draw[black](345)--(2345);
\draw[black](1234)--(12345);
\draw[black](2345)--(12345);


  \node (0) at (3,1) {\textcolor{black}{$\emptyset$}};

  \node (3) at (3,3) {{$\bf 3$}};
  \node (34) at (5,5) {$3{\bf 4}$};
  \node (23) at (1,5) {${\bf 2}3$};
  
  \node (234) at (3,7) {$234$};
  \node (123) at (-1,7) {${\bf 1}23$};
  \node (345) at (7,7) {$34{\bf 5}$};

  \node (1234) at (1,9) {$1234$};
  \node (2345) at (5,9) {$2345$};

  \node (12345) at (3,11) {$12345$};

\draw[black](0)--(3);
\draw[black](3)--(34);
\draw[black](3)--(23);
\draw[black](23)--(234);
\draw[black](34)--(234);
\draw[black](34)--(345);
\draw[black](23)--(123);
\draw[black](123)--(1234);
\draw[black](234)--(1234);
\draw[black](234)--(2345);
\draw[black](345)--(2345);
\draw[black](1234)--(12345);
\draw[black](2345)--(12345);

\end{tikzpicture}\caption{(Left) The lattice $\ml$ of an LCI model. (Note that $\emptyset$ is always the unique minimum element of $\mathcal{L}$.) (Right) The monomials $u_I$ associated to each element $I$ of $\mathcal{L}$. Note that $z_i$ corresponds to the lattice point which is the smallest subset containing $i$ (marked in {\bf bold} in $\mathcal{L}$).}\label{fig:lattice}
 \end{center}
\end{figure}
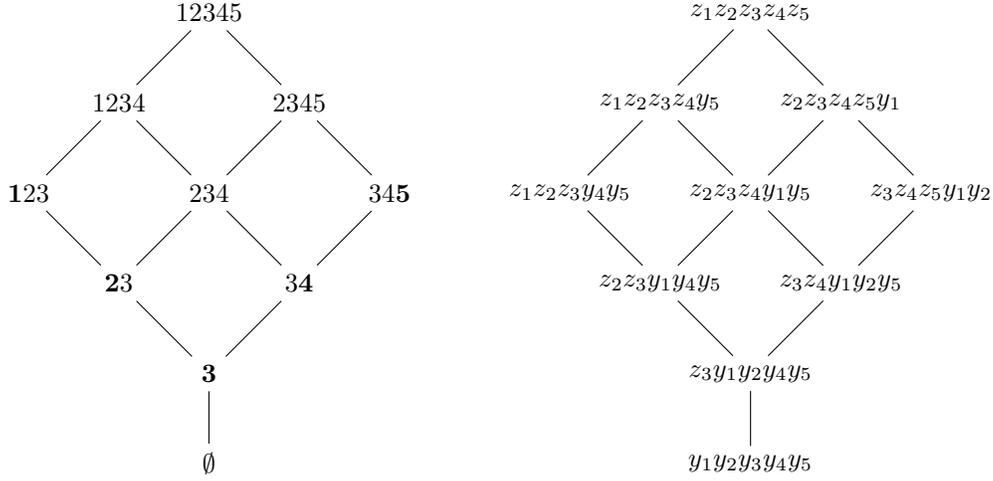

\begin{figure}
 \begin{center}
\begin{tikzpicture} [scale = .5]

  \node (3) at (15,3) {${\bf 3}$};
  \node (34) at (17,5) {$3{\bf 4}$};
  \node (23) at (13,5) {${\bf 2}3$};
  \node (123) at (11,7) {${\bf 1}23$};
  \node (345) at (19,7) {$34{\bf 5}$};

\draw[black](3)--(34);
\draw[black](3)--(23);
\draw[black](34)--(345);
\draw[black](23)--(123);


  \node (3) at (3,3) {$\bf 3$};
  \node (34) at (6,4) {${\bf 4}$};
 \node (23) at (0,4) {${\bf 2}$};
  \node (123) at (-1,7) {${\bf 1}$};
  \node (345) at (7,7) {${\bf 5}$};

\draw[->](3)--(123); 
\draw[->](3)--(345); 

\draw[->](3)--(34);
\draw[->](3)--(23);
\draw[->](34)--(345);
\draw[->](23)--(123);

\end{tikzpicture}\caption{(Left) The TDAG of the above LCI model. 
(Right) The joint irreducible poset $\ji$.
}\label{fig:TDAG}

\end{center}
\end{figure}
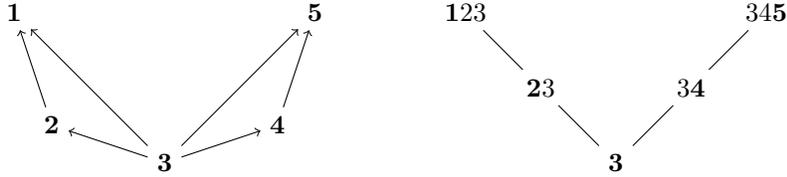


Before we explain LCI models in more detail, we recall the notion of join-irreducible elements of a lattice.
\begin{Definition}\label{def:join_irreducible}
Let $(\ml,\wedge,\vee)$ be a finite distributive lattice.
An element $j\in\ml$ is called \textit{join-irreducible} if $j=\ell_1\vee\ell_2$ for some $\ell_1,\ell_2\in \ml$ implies  $j=\ell_1$ or $j=\ell_2$. Denote by $\ji$ the set of join-irreducible elements in $\ml$. 
\end{Definition}

\begin{Example}
In examples, in order to simplify our notation we denote the set $\{i_1,i_2,\ldots, i_t\}$ by $i_1i_2\cdots i_t$, where $i_1 < i_2 < \cdots < i_t$. Figure~\ref{fig:lattice} depicts the lattice generated by the sets $123,234,345$ and the corresponding monomials $u_I$ for each lattice element $I$. The non-empty join-irreducible elements of $\mathcal{L}$ are $3,23, 34, 123$, and $345$. The corresponding join-irreducible post is depicted in Figure~\ref{fig:TDAG} (right). We have marked an element in bold, when it is the first time that it shows up in $\mathcal{L}$ (looking at $\mathcal{L}$ from bottom to top).
The corresponding TDAG is shown in left that will be further explained in Section~\ref{Sec:TDAG}.
\end{Example} 
\begin{Remark}\label{rem:g_I}
Note that, every element $I$ of $J(\mathcal{L})$ is uniquely determined by an element $i$ for which $I$ is the smallest lattice element containing it. Hence, we may use the notation $g_i$ for its corresponding monomial $u'_I=\prod_{j\leq i} z_j$. For example, Figure~\ref{fig:lattice} shows that $123$ is the smallest element of the lattice containing $1$, so the monomial $z_1z_2z_3$ is denoted by $g_1$. Similarly, we have $g_2=z_2z_3$, $g_3=z_3$, $g_4=z_3z_4$ and $g_{5}=z_3z_4z_5$.
\end{Remark}

Note that the partial order on $\ml$ induces a partial order on $\ji$. 
Hence, taking the poset $\ji$, we can 
define the Hibi ideal of $\ji$ or simply the Hibi ideal of the lattice $\mathcal{L}$. %
More precisely, Birkhoff's fundamental structure Theorem \cite[Theorem 3.4.1]{volume1986wadsworth} guarantees that the lattices $(\mathcal{O}(\ji),\cap,\cup)$ and $(\ml,\cap,\cup)$ are isomorphic. One fundamental result in this context is that the Hibi ideal of $\ml$ is generated by the following relations (see, e.g.,~\cite[Theorem 10.1.3]{herzog2011monomial}):
\[
p_Ip_J-p_{I\cap J}p_{I\cup J}\quad \text{with} \quad I\not\subseteq J\quad \text{and}\quad J\not\subseteq I\quad (I,J\in\mathcal{L}).
\]

We can now interpret the terms $p_I, p_J, p_{I \cup J},p_{I\cap J}$ as marginal probability distributions associated with their appropriate margins. The only bit of additional notation we need is the conditioning sign that we will develop in the next section. Thus the above binomials in the Hibi ideal give us the conditional independences in \eqref{CI}.

\section{Transitive directed acyclic graphs (TDAGs)}\label{Sec:TDAG}
\subsection{The probabilistic interpretation of the join irreducible elements 
}\label{sec:TDAG_L}

We start by giving the formal definition of the TDAG associated to a distributive lattice.

\begin{Definition}[TDAG of a lattice]\label{def:TDA_LG}
Let $\ml$ be a distributed lattice whose elements are corresponding to a collection of subsets of $[m]$. Assume that $\ml$ has $n$ join-irreducible elements. The TDAG associated with $\ml$ is a transitive directed acyclic graph $G([n],E)$ where the vertex $i$ represents a join-irreducible element of $\ml$ containing $i$ for the first time (ordering the join-irreducible elements from bottom to top) and a directed edge $e = \{i \rightarrow j\}$ belongs to $E$ if and only if the join-irreducible elements $i$ and $j$ are comparable in the lattice and $i<j$. 
\end{Definition}
\begin{Example}
The associated TDAG of the lattice $\mathcal{L}$ in Figure~\ref{fig:lattice} is depicted in Figure~\ref{fig:TDAG} (left).
\end{Example}

Figures~\ref{fig:lattice} and \ref{fig:TDAG} show the relations between the LCI lattice, the set of monomials associated with the lattice and the
corresponding TDAG of the lattice. 
Up to this point the development has been purely algebraic. 
Now we interpret the variables $p_J$, 
$z_i$ and $y_k$ in probabilistic terms.

\medskip

We can give a formal meaning, or interpretation, of the $ z_i$ as
$$z_i = p_{\{j: j \leq i \} | \{j: j < i\}},$$
where $j \leq i$ is in the lattice sense. The notation means that we condition $X_i$ on the ``ancestors" of $X_i$, with respect to the lattice.
For instance, let us consider the example of the set $123$ in the lattice $\mathcal{L}$ in Figure~\ref{fig:lattice}. This has an ascending chain of subsets leading to:
$$ 3 \subset 23 \subset 123,$$
and a  corresponding chain of monomials:
$$ z_3,\ z_2 z_3,\ z_1z_2z_3.$$  
The interpretation of the $z_i$ is as the terms in the development of marginal $p_{123}$ in terms of conditions. Thus,
$$p_{123} = p_3 p_{23|3} p_{123| 23},$$
and we make the mapping
$$z_3 \leftrightarrow p_3, \; z_2 \leftrightarrow  p_{23|3}, \; z_1 \leftrightarrow p_{123|23}.$$
Similarly, let us check $p_{234}$. Mapping to the $z_i$ we have
$$ z_2 z_3 z_4  \rightarrow \left(\frac{p_{23}}{p_3}\right)  p_3 \left( \frac{p_{34}}{p_3}\right) = \frac{p_{23} p_{34}}{p_3} = p_{234}.$$
Note that the last equality derives from the equation $p_{23}p_{34}=p_3 p_{234}$ arising from the corresponding Hibi ideal.
More generally, for every saturated chain of subsets 
\[
J_0\subset J_1\subset \cdots\subset J_r=J
\]
we have that
\begin{eqnarray*}
p_{J} & = & p_{J_0} p_{J_1|J_0} p_{J_2|J_1}\cdots  p_{J|J_{r-1}}\\
             & =  & z_{J_0} z_{J_1\backslash J_0} z_{J_2\backslash J_1} \cdots z_{J_r\backslash J_{r-1}} .
\end{eqnarray*}
We should note that in our development we increment subsets with a single ``$i$", but in general the increment can be larger. In those cases, we can extend the definition of the index set, in an obvious way.  We now state and prove the relationship between the marginal and conditional probabilities.

\begin{Theorem}
Consider a probability distribution with a set of positive marginals 
$p_J = \prod_{i\in J} z_{i}$ where the collection of sets $J$ forms a distributive lattice satisfying the Hibi relations. Then:
\begin{itemize}
    \item[{\rm(i)}] The quantities $z_i$ are conditional probabilities $z_i = p_{\{j: j \leq i \} | \{j: j < i\}}$.

\item[\rm{(ii)}] The Hibi relations give all conditional independent statements associated with the lattice.
\end{itemize}
\end{Theorem}
\begin{proof}
The conditions satisfied by the Hibi ideal can be rearranged
as
$$ p_{ I \cup J} = \frac{p_{I} p_J}{p_{I \cap J}}.$$
Now assume that we have a probability distribution with positive marginals supported on the lattice. Then the logarithm $\phi(I) = \log p_I$ induces a valuation
\begin{equation}\label{eq:phi_union}
    \phi(I\cup J) = \phi(I) + \phi(J) - \phi(I \cap J).
\end{equation}
We then invoke the theory of valuations on a distributive lattice and claim that the  conditional independence relations arise from M\"obius inversion on the lattice, essentially inclusion-exclusion formula for $\phi(I)$. More precisely, considering the saturated chains of subsets 
\[
I\cap J=J_0\subset J_1\subset \cdots\subset J_r=J\quad\text{and}\quad I\cap J=I_0\subset I_1\subset \cdots\subset I_s=I,
\]
the Equation
\eqref{eq:phi_union} translates to 
\begin{eqnarray*}
p_{I\cup J} & = & \frac{p_{I}  p_{J}}{p_{I\cap J}}\\
             &  = & \frac{p_{I}}{p_{I\cap J}} \frac{p_{J}}{p_{I\cap J}} p_{I\cap J}\\
             & =  &  \left(\frac{p_{I}}{p_{I_{s-1}}}\cdots \frac{p_{I_1}}{p_{I\cap J}} \right)
              \left(\frac{p_{J}}{p_{J_{s-1}}}\cdots \frac{p_{J_1}}{p_{I\cap J}} \right)
               p_{I\cap J}\\
             & =  & \prod_{i=1}^s z_{I_{i}\backslash I_{i-1}}\ \prod_{j=1}^r z_{J_{j}\backslash J_{j-1}} z_{I\cap J}.  \quad\quad\quad\quad\quad\quad\quad\quad\quad
             \mbox{\qedhere}
\end{eqnarray*}
\end{proof}
\subsubsection{The complementary lattice of $\mathcal{L}$ and the reverse TDAG}\label{sec:dual}
Given a lattice $L$ on a collection of subsets $\mathbb{I}$ of $[n]$, we denote its complementary lattice with $\mathcal{L}^*$ whose elements are $[n]\backslash I$ (for $I\in \mathbb{I}$) and are reversely ordered comparing to $\mathcal{L}$. 
In the exact same way that we have defined the entities for $\mathcal{L}$, we can associate the dual entities $q_J$ and the $y_i$ for the complementary lattice $\mathcal{L}$. If we want to think
of each arrow as a generalisation of time, we might say we condition on the ``descendants" of $i$.  
In the same way, the variable $y_i$ has an interpretation as the conditional random variable for the distribution of random variables $Y_i$ defined by the margins $q_J= \frac{p_{[n]}}{p_{[n] \setminus J}}$. Given the distribution defined by the $q_J$ margins:
$$y_i =q_{\{j: j \geq i\}\mid\{ j: j > i\}},$$
the question remains as what ideal do we obtain if we eliminate the $y_i$ from the ideal
$$\langle q_J-\prod_{i\in {J}}y_i:\ J\in \mathcal{O}(Q)\rangle.$$
This has an interesting probabilistic interpretation. In \S\ref{sec:Alex}, we will see that it relates to reversing the arrows in the associated TDAG of $\mathcal{L}$. 
We will also connect this graph operation with the Alexander duality of the corresponding ideal.

\medskip

Just as we associate a marginal random variable $X_I$ with every index set $I \in \mathbb I$, so we associate a random variable $Y_J$ with every $J$ in the complementary lattice $\mathcal{L}^*$ on the subsets $\mathbb J$. The random variable $Y_J$ is interpreted as having the conditional distribution of $X_J$ given $X_{[n] \setminus J}$. This has the distribution:
\begin{eqnarray}\label{eq:q_J}
q_J = \frac{p_{[n]}}{p_{[n] \setminus J}}.
\end{eqnarray}
The duality is clear: writing $I = [n] \setminus J$ we have
$$p_I q_J = p_{[n]}.$$
For the complementary lattice of Figure~\ref{fig:lattice}, it can be shown that $Y_{12}$ and $Y_{45}$ are independent, as
\begin{eqnarray*}
q_{12}q_{45} & = & \frac{p_{12345}}{p_{345}} \frac{p_{12345}}{p_{123}} \\
                      & = & \frac{p_{12345}^2}{p_{12345} p_3} \\
                      & = & \frac{p_{12345}}{p_3} \\
                       & = & q_{1245},
                      \end{eqnarray*}
giving the factorisation for the required independence.

The Gaussian case reflects the duality well. If one considers the projection operator $P_J$ associated with the Gaussian random variable $X_J$ then the Gaussian random variable associated with the random variable $Y_{[n] \setminus I}$ has the orthogonal projector
\begin{eqnarray}\label{eq:Complement}
Q_{[n] \setminus I} = I_n - P_I.
\end{eqnarray}
In our running example, we can see that
$$Q_{12} Q_{45} = (I-P_{345})(I-P_{123}) = I- P_{345} - P_{123} + P_{3}=0,$$
hence we recognise the conditional independence $Y_{12} \bigCI Y_{45} | Y_{3}$.

\subsection{Log-linear models for margins}
With a view to statistical modelling, we would like to point out that the $z$-representation leads to log-linear models, not only of the more familiar form for a  single probability $p(x)$ but also for all the margins. 
If for a margin for the index set $I$ we have the product form
$$p_I(x) = \prod_{i \in I} z_i(x),$$
and 
$$ \log p_I(x_I) = \sum_{i \in I} u_i(x_I),$$
where $u_i(x_I) = \log z_i(x_I)$, then this can be written in exponential form as
$$p_I(x) = \exp \left\{ \sum_{i \in I} u_i(x_I) \right\}.$$
Because of the dependence on $x_I$ this should be referred to as a {\em local} model. In statistical modelling, one may start with
a class of functions $\mathcal U = \{u(x_I)\}$, which defines the properties of our model and extend the class in a parametric way by introducing
parameters $\theta_{i,I}$ and writing
$$ \log p_I(x) = \exp \left\{ \sum_{i \in I} \theta_{i,I} u_i(x_I) \right \}.$$
Similar models can be written down for the $q$-margins.  We do not discuss statistical analysis in this paper, but we refer to 
\cite{massam1998estimation} for an overview in the Gaussian case.  Maximum likelihood estimation and other statistical
issues in the general case remain open, but we note the promising contribution from \cite{beerenwinkel2007conjunctive}. 

\subsection{Alexander duality and a special  bipartite graph} \label{sec:Alex}

We promised that the algebraic representation would yield certain results which may prove useful in the probability theory. Let $\mathcal{L}$ be a distributive lattice with join-irreducible poset $Q=J(\mathcal{L})$, see  Definition~\ref{def:join_irreducible}. The significance of Hibi ideals is that the Alexander dual of the monomial ideal $M_Q=\langle u_I:\ I\in \mathcal{O}(Q)\rangle$ can be interpreted as the edge ideal of a special bipartite graph. To be precise, for a poset $Q$ on $ 1,\ldots,n$, if we define the bipartite graph $G$ on the vertex set $\{z_1,\ldots,z_n,y_1,\ldots,y_n\}$ by saying that $\{z_i,y_j\}$ is an edge of $G$ if and only if $i\leq j$ in $Q$, then the Alexander dual of the ideal $M_Q$ is the edge ideal of $G$. 
It turned out that bipartite graphs obtained in this way are exactly the Cohen--Macaulay bipartite graphs \cite[Lemma~3.1]{Herzog}. 


We first recall the notion of Alexander duality for squarefree monomial ideals which naturally arise in this context (see \cite[Def. 5.20]{MS05}). 
\begin{Definition}[Alexander dual]
Let $K[{\bf x}]=K[x_1,\ldots,x_n]$ be a polynomial ring over a filed $K$ in $n$ variables.
  Let $M = \langle {\bf x}^{a_1},\ldots,{\bf x}^{a_r}\rangle$ be a squarefree monomial ideal in $K[{\bf x}]$, i.e.~it is generated by monomials ${\bf x}^{a_i}$ for $i=1,\ldots,r$, where every coordinate of $a_i=(a_{i,1},\ldots,a_{i,n})$ is either $0$ or $1$. The Alexander dual of $M$
is the monomial ideal $M^\ast =\underline{m}^{a_1}\cap\cdots\cap\underline{m}^{a_r}$, where 
$\underline{m}^{a_i}=\langle x_j:\ a_{i,j}=1\rangle$.
\end{Definition}

\begin{Definition}[Edge ideal]
Let $G$ be a directed graph on the vertices $1,\ldots,n$. Consider the polynomial ring $K[{\bf z,y}]=K[z_1,\ldots,z_n,y_1,\ldots,y_n]$ on $2n$ variables, where $z_i,y_i$ corresponds to the vertex $i$ of $G$. The edge ideal of $G$ is the ideal $I(G)=\langle z_iy_j:\ \{i\rightarrow j\}\ \text{is\ an\ edge\ of}\ G \rangle$.
\end{Definition}
\begin{Example}
Consider our running example depicted in Figures~\ref{fig:lattice} and \ref{fig:TDAG}. The ideal $M_Q$ is: 
$$M_Q=\langle y_1 y_2 y_3 y_4 y_5, z_3 y_1 y_2 y_4 y_5, z_2z_3 y_1 y_4 y_5, z_3z_4 y_1 y_2 y_5,  z_1z_2z_3 y_4 y_5,  z_2z_3 z_4y_1 y_5, 
$$
$$z_3z_4z_5 y_1 y_2, z_1z_2z_3 z_4 y_5,z_2 z_3 z_4 z_5 y_1, z_1 z_2 z_3 z_4 z_5 \rangle.$$
The Alexander dual of $M_Q$ is generated by the degree 2 monomials as follows:
$$M_Q^{\ast}=\langle z_3 y_2, z_3 y_1,z_2 y_1, z_3 y_4, z_3 y_5, z_4 y_5, \\  z_1 y_1, z_2 y_2, z_3y_3,z_4y_4,z_5y_5 \rangle,$$
which is the edge ideal of the TDAG in the left hand of Figure~\ref{fig:TDAG}, where we have removed the ``loops" corresponding to the monomials $z_i y_i$. 
\end{Example}
More formally, we have the following corollary \cite[Lemma~3.1]{Herzog} adopting our notation. 
\begin{Corollary}\label{cor:Herzog}
Let $\mathcal{L}$ be a distribute lattice with the joint irreducible poset $Q=J(\mathcal{L})$.
Then the Alexander dual of the ideal $M_Q=\langle u_I:\ I\in \mathcal{O}(Q)\rangle$ is the monomial ideal $M_Q^\ast$ generated by $z_iy_j$ where $i\leq j$ in $Q$. In particular, $M_Q^\ast$ is the edge ideal of the TDAG on the vertices $1,\ldots,m$ with the edge set $\{i\rightarrow j:\ i\leq j \ \text{in}\ Q\}$.
\end{Corollary}

\medskip

This analysis gives a pleasing duality which we state in the following lemma.

\begin{Lemma}
Let $\mathcal{L}$ be a distribute lattice with the joint irreducible poset $Q=J(\mathcal{L})$. Let $I(i) = \{j: j \leq i\}$ and $J(j) = \{k: k \geq j\}$ for $i,j$ in $Q$. Following the notation of \S\ref{sec:dual}, let $p_I$ denote the probabilities associated to $\mathcal{L}$ with dual entities $q_J$ associated to the complementary lattice of $\mathcal{L}$. Then, for every $i\leq j$ we have that:
$$ \frac{p_{I(j)}}{p_{I(i)}} = \frac{q_{J(j)}}{q_{J(i)}}.$$
\end{Lemma}
\begin{proof}
Let $G=G([n],E)$ be the corresponding TDAG of $\mathcal{L}$. By Corollary~\ref{cor:Herzog}, for every vertex $i$ of $G$, $\{j\rightarrow i\}$ is an edge if $j$ is in the set $I(i)$. 
Similarly, for the complementary lattice of $\mathcal{L}$, $\{k\rightarrow j\}$ is an edge in the corresponding TDAG if $k$ is in the set $J(j)$. In other words, the corresponding TDAG is the reverse of $G$, i.e.~its edges are obtained by reversing the edges of $G$. Then, note that for any member $p_I q_J$ for disjoint sets $I,J$ with $I \cup J = [n]$ we have:
$$ p_{I(i)} q_{J(j)} = p_{I(j)} q_{J(i)} = p_{[n]}.$$
Thus if $i \leq j$ we have:
$$ \frac{p_{I(j)}}{p_{I(i)}} = \frac{q_{J(j)}}{q_{J(i)}},$$
giving a symmetry in forward and backward conditional independence on the TDAG and the reverse TDAG. 
\end{proof}

As we have seen above, the Alexander dual framework unifies the forward and backward TDAGs analogously to its role in
the paths and cuts duality of algebraic approaches to system reliability \cite{giglio2004monomial, mohammadi2016algebraic}. Moreover, it provides us with various computational toolkits on Macaulay2 \cite{M2} such as the ``Posets.m2" package.

\begin{figure}
\begin{center}
\begin{tikzpicture} [scale = .6]
\node (11) at (1,7) {$X_{11}$};
\node (12) at (3,7) {$X_{12}$};
\node (13) at (5,7) {$X_{13}$};
\node (21) at (1,5) {$X_{21}$};
\node (22) at (3,5) {$X_{22}$};
\node (23) at (5,5) {$X_{23}$};
\node (31) at (1,3) {$X_{31}$};
\node (32) at (3,3) {$X_{32}$};
\node (33) at (5,3) {$X_{33}$};
\draw[black][->](11)--(12);
\draw[black][->](12)--(13);
\draw[black][->](21)--(22);
\draw[black][->](22)--(23);
\draw[black][->](31)--(32);
\draw[black][->](32)--(33);
\draw[black][->](21)--(12);
\draw[black][->](21)--(13);
\draw[black][->](21)--(32);
\draw[black][->](21)--(33);
\draw[black][->](22)--(13);
\draw[black][->](22)--(33);

\node (11) at (10,3) {${\bf 11}$};
\node (21) at (15,3) {${\bf 21}$};
\node (31) at (20,3) {${\bf 31}$};
\node (12) at (10,5) {$11,21,{\bf 12}$};
\node (22) at (15,5) {$21,{\bf 22}$};
\node (32) at (20,5) {$21,31,{\bf 32}$};
\node (13) at (10,7) {$11,21,12,22,{\bf 13}$};
\node (23) at (15,7) {$21,22,{\bf 23}$};
\node (33) at (20,7) {$21,31,22,32,{\bf 33}$};
\draw[black](11)--(12);
\draw[black](21)--(22);
\draw[black](31)--(32);
\draw[black](12)--(13);\draw[black](22)--(13);
\draw[black](22)--(23);
\draw[black](32)--(33);\draw[black](22)--(33);
\draw[black](21)--(32);\draw[black](21)--(12);


\end{tikzpicture}
\caption{(Left) Time series 
for $t=1,2,3$. To simplify the graph, the edges $X_{i1}\rightarrow X_{i3}$ (for $i=1,2,3$) are not depicted. (Right) The join-irreducible poset o the corresponding lattice. The variable $z_{ij}$ corresponds to the point which is the smallest subset containing $ij$ (marked in {\bf bold}).}
\label{fig:time}
\end{center}
\end{figure}
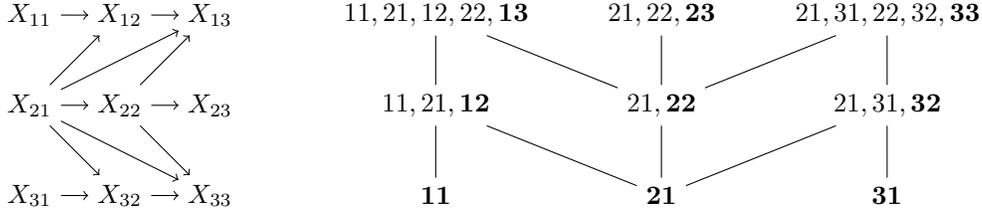

\section{Applications}

\subsection{Time series}
We now give a somewhat larger example showing how to apply the LCI principle to three univariate time series $\{X_{1t}, X_{2t}, X_{3t}\}$. We define a lattice conditional independence model
which 
describes the conditional independence between the first time series up to time $t$ and the third time series up to time $t$, given the second time up to time $t-1$. Moreover, in the general case this holds for all $t=1, 2, \ldots$. For $t= 1$, we may simply assume that $X_{11}$, $X_{21}$ and $X_{31}$ are independent (for simplicity we have omitted cross-sectional conditional independence). 


\begin{Example}\label{ex:time}
Figure~\ref{fig:time} (left) shows the TDAG scheme for three time points. We have simplified the labeling in the middle to avoid the double suffix. The corresponding lattice $\mathcal{L}$ has 45 elements indexed by the following subsets (ordered by inclusions):
\begin{itemize}
    \item[ ] $\scriptstyle{\emptyset, 
 \{{11}\}, 
 \{{21}\}, 
 \{{11}  , {21}\}, 
 \{{31}\}, 
 \{{11} , {31}\}, 
 \{{21} , {31}\}, 
 \{{11} , {21} , {31}\},
 \{{11} , {21} , {12}\}, 
 \{{11} , {21}  , {31}  , {12}\}, 
 \{{21} , {22}\}, }
 \{{11} , {21}  , {22}\}, 
 \{{21}  , {31}  , {22}\}, 
  $
\item[ ]  $\scriptstyle{\{{11}  , {21}  , {31}  , {22}\}, \{{11}  , {21}  , {12}  , {22}\}, 
 \{{11}  , {21}  , {31}  , {12}  , {22}\},
 \{{21}  , {31}  , {32}\}, 
 \{{11}  , {21}  , {31}  , {32}\}, 
 \{{11}  , {21}  , {31}  , {12}  , {32}\}, 
 \{{21}  , {31}  , {22}  , {32}\}, 
 } $
 \item[ ]
$\scriptstyle{
  \{{11}  , {21}  , {31}  , {22}  , {32}\}, 
 \{{11}  , {21}  , {31}  , {12}  , {22}  , {32}\},
 \{{11}  , {21}  , {12}  , {22}  , {13}\}, 
 \{{11}  , {21}  , {31}  , {12}  , {22}  , {13}\},
 \{{11}  , {21}  , {31}  , {12}  , {22}  , {32}  , {13}\}, 
 \{{21}  , {22}  , {23}\}, }
$

\item[ ] $\scriptstyle{
  \{{11}  , {21}  , {22}  , {23}\},  
 \{{21}  , {31}  , {22}  , {23}\}, 
 \{{11}  , {21}  , {31}  , {22}  , {23}\},
 \{{11}  , {21}  , {12}  , {22}  , {23}\}, 
 \{{11}  , {21}  , {31}  , {12}  , {22}  , {23}\}, 
 \{{21}  , {31}  , {22}  , {32}  , {23}\},
  } $
  
  \item[ ] $
 \scriptstyle{
  \{{11}  , {21}  , {31}  , {22}  , {32}  , {23}\}, 
 \{{11}  , {21}  , {31}  , {12}  , {22}  , {32}  , {23}\}, 
 \{{11}  , {21}  , {12}  , {22}  , {13}  , {23}\}, 
\{{11}  , {21}  , {31}  , {12}  , {22}  , {13}  , {23}\},
 \{{11}  , {21}  , {31}  , {12}  , {22}  , {32}  , {13}  , {23}\}, 
  }$
 \item[ ]
 $\scriptstyle{
 \{\{{21}  , {31}  , {22}  , {32}  , {33}\}, 
 \{{11}  , {21}  , {31}  , {22}  , {32}  , {33}\}, 
 \{{11}  , {21}  , {31}  , {12}  , {22}  , {32}  , {33}\},
 {11}  , {21}  , {31}  , {12}  , {22}  , {32}  , {13}  , {33}\}, 
 \{{21}  , {31}  , {22}  , {32}  , {23}  , {33}\}, 
}$

\item[ ]
$\scriptstyle{\{{11}  , {21}  , {31}  , {22}  , {32}  , {23}  , {33}\},  
 \{{11}  , {21}  , {31}  , {12}  , {22}  , {32}  , {23}  , {33}\}, 
 \{{11}  , {21}  , {31}  , {12}  , {22}  , {32}  , {13}  , {23}  , {33}\}}.$
\end{itemize}
Figure~\ref{fig:lattice_time_series} (left) shows the corresponding poset $Q=J(\mathcal{L})$ of join-irreducible elements. 
To verify that the TDAG is recaptured using the Alexander dual, for $Q$ in Figure~\ref{fig:lattice_time_series} (left), consider the ideal
$$M_Q=\langle y_{11} y_{21} y_{31} y_{12} y_{22} y_{32} y_{13} y_{23} y_{33}, z_{11} y_{21}y_{31}y_{12}y_{22}y_{32}y_{13}y_{23}y_{33},  
$$$$ z_{21}z_{31}z_{22}z_{32}z_{23}z_{33} y_{11}y_{12}y_{13},\ldots ,z_{11}z_{21}z_{31}z_{12}z_{12}z_{32}z_{13}z_{23}z_{33}\rangle$$
generated by 45 monomials associated to the subsets above. Then the Alexander dual of $M_Q$ is:
$$M_Q^\ast=\langle z_{11} y_{12}, z_{21} y_{12}, z_{21} y_{22},  z_{21} y_{32},  z_{31} y_{32},  z_{11} y_{13}, z_{21} y_{13}, z_{12} y_{13}, z_{22} y_{13}, z_{21} y_{23}, z_{22} y_{23}, z_{21} y_{33},$$ $$  z_{31} y_{33}, z_{32} y_{33},  z_{22} y_{33},  z_{11}y_{11},z_{21}y_{21},z_{31}y_{31},z_{12}y_{12},z_{22}y_{22},z_{32}y_{32},z_{13}y_{13},z_{23}y_{23},z_{33}y_{33} \rangle.$$
Now, following Corollary~\ref{cor:Herzog} we obtain the edge representation of the time series in Figure~\ref{fig:time} (left) by removing the ``loops" corresponding to the monomials $z_{ij} y_{ij}$ in $M_Q^\ast$.
\end{Example}

To preserve the conditional independence 
as time progresses, we need to show how to update the TDAG and the associated join-irreducible poset. If we add another time layer, then the new join-irreducible poset has one more layer with the index sets
$$\{11, 12, 13, 14, 21, 22, 23\}, \{21,22, 23, 24\},\{21, 22, 23, 31, 32, 33, 34\}.$$ 
These are the generators which appear at the top of the previous poset for $t=1,2,3$:
$$\{11, 12, 13, 21, 22\}, \{21, 22, 23\},\{21, 22, 31, 32, 33\}.$$
It should be clear, then, that the rule for updating the lattice is just to update the ancestral sets at each time point. 
But this also has a rule in terms of the $z_{ij}$. More precisely, we have that:
\begin{eqnarray*}
p_{\{11, 12, 13, 14, 21, 22, 23\}} & = & z_{23} z_{14} \cdot p_{\{11, 12, 13, 21, 22\}} \\
p_{\{21,22, 23, 24\}} &  = &  z_{24} \cdot p_{\{21, 22, 23\}}\\
p_{\{21, 22, 23, 31, 32, 33, 34\}} & = & z_{23} z_{34} \cdot p_{\{21, 22, 31, 32, 33\}}
\end{eqnarray*}
Thus the rule is to update the lattice generators with the conditional probabilities corresponding to the new arrows. In time series language, $z_{23},z_{14}, z_{24}, z_{34}$  can be considered as the innovations corresponding to the conditional independence structure or, to put it more strongly, the innovations required to {\em preserve} the conditional independence. All new lattice-based conditional independence statements are obtained by expanding with subsets, inclusions and intersections in the usual way. 

\subsubsection{Transfer functions}
The $z$-products, which we have seen are strings of conditional probabilities, facilitate the updating of time series models, by the use of transfer functions. We only give a sketch, here. We suggest that an appropriate approach is to use
the generators $g_i$ in the sense of Remark~\ref{rem:g_I}, where $i$ uniquely determines the subset $I$ which is the smallest join-irreducible element containing $i$). For example, as depicted in Figure~\ref{fig:time} (right), the smallest set containing $13$ is  $\{11, 21, 12, 22, 13\}$. Let us denote the corresponding monomial with $g_{13}$ as follows. 
Thus, the transfer functions can be based on updating the generators. These are products, or sets of products, but their logarithms give linear functions, creating a log-linear time series.

\medskip

More precisely, in our example in Figure~\ref{fig:time} 
the generators corresponding to the top-layer at time $t=3$ are given by
\begin{eqnarray*}
g_{13} & = & z_{11} z_{21}z_{12} z_{22} z_{13}, \\
g_{23} & = & z_{21} z_{22} z_{23}, \\
g_{33} & = & z_{21} z_{31} z_{22} z_{32} z_{33}. 
\end{eqnarray*}
Taking the same LCI structure advancing to $t=4$, as depicted in Figure~\ref{fig:lattice_time_series}, we multiple entrywise by a new vector, namely
\begin{eqnarray*}
{u_{1t}} & = &  z_{14}z_{23}, \\
{u_{2t}} & = & z_{24}, \\
{u_{3t}} & = & z_{24}z_{34}. 
\end{eqnarray*}
We can express the updating symbolically, with obvious notation, by
$${g_{i,t+1} = g_{it}  \circ u_{it},}$$
for every $i=1,\ldots,3$, where $\circ$ is the Schur (entrywise) product.
Taking logarithms we write:
$$\log( g_{i,t+1}) = \log(g_{it})  + \log( u_{it}).$$
Now, using the shift operator $S$ from time series we can write this as
$$S \cdot \log( g_{it}) = \log(g_{it})  + \log( u_{it}),$$
leading to 
$$ (S-1) \log(g_{it}) = + \log( u_{it}), $$
and formally
\begin{eqnarray*}
\log(g_{it}) & = & (S-1)^{-1} \log(u_{it}),\\
              & = & T(1-T)^{-1} \log(u_{it}),
\end{eqnarray*}
Here $T = S^{-1}$ is the reverse (backward in time) shift operator. Then, by expanding $$(1-T)^{-1} = 1 + T + T^2 + \cdots ,$$ and with suitable initial conditions, we can develop the process in terms of the ``innovations" $u_{it}$, noting that 
$T \log(u_{it}) = \log (u_{i,t-1}).$

\subsection{Information flow}
The Shannon information associated with the margin $I$ is
$$H(I) = \mbox{E} \{\log (p_I( X_I))\},$$
where the expectation $\mbox{E}$ is with respect to the full joint distribution (although we could compute the expectation with respect to the marginal random variable $X_I$).
Taking logarithms and then expectation of the Hibi relations (assuming positivity of the probabilities) we have
$$p_I(X)p_J(X)=p_{I\cap J}(X)p_{I\cup J}(X),$$
and we  see that $H$ is a {\em valuation} on our distributive lattice: for all pairs $I,J \in \mathcal{L}$ we have
$$ H(I \wedge J) + H(I \vee J) = H(I) + H(J). $$
The theory of valuations on distributive lattices is famously developed by Rota \cite{rota1964foundations,kung2009combinatorics}.
It is important to note that, whereas $\phi(I) = \log \phi_I$ coming from a probability evaluation in Theorem~\ref{thm:energy} depends on the argument $X_I$, that is a local evaluation, the expectation in the evaluation of Shannon information removes the random variable $X_I$, leaving only the index  $I$. Thus we may write $H(I)$ as a numerical quantity attached directly to the margin $I$.
For example, we can again use simple suffix notation. Thus, for simple conditional expectation we write:
$$H({123})  = H({23}) + H(34) - H(3).$$
Moreover, we have a TDAG representation of change in information. Returning to our running TDAG left branch $3 \subset 23 \subset 123$ from Figure~\ref{fig:lattice}, we have
\begin{eqnarray}
H(123) & = & H(3)  + \{H({23}) - H(3)\} + \{H({123})- H({23})\} .
\end{eqnarray} 
The three terms on the right hand side correspond respectively to the arrows in our TDAG: $3 \rightarrow 1, 3 \rightarrow 2, 2 \rightarrow 1$ (see Figure~\ref{fig:TDAG}). We see that the TDAG
describes a (cumulative) additive model for information. The transitive closure edge $3 \rightarrow 1$ simply contributes $H(3)$ to the total information at node $1$.  

The increments along each edge have an interpretation in terms of the Shannon result for disjoint index sets $I,J$:
$$H(I \cup J) = H(I) + \mbox{E}_I \{ H(J | I) \}.$$
Thus, in our example we can write
\begin{eqnarray}
H({123}) & = & H(3)  + E_3{H({2|3})} + E_{23}H({1|23}).
\end{eqnarray} 
This interpretation of the TDAG as representing an information flow diagram, holds generally for LCI models, and is reminiscent of various types of compartmental flow models, for fluids, or perhaps a better analogy is for energy, recalling that Shannon entropy is the negative of Shannon information. Essentially, it arises from the underlying Markov structure.

We express this result as an informally stated theorem, without proof.
\begin{Theorem}\label{thm:energy}
The TDAG of an LCI model provides an additive information model in which to each edge
$ (i \rightarrow j)$ is associated energy increment:
$H(I) - H(J)$ where $I = J \cup \{i\}$.      
\end{Theorem}
As we mentioned in the discussion above Theorem~\ref{thm:energy}, one property of valuations on distributive lattices is the existence of inclusion-exclusion lemma.  This allows us to express the information for any union of index sets in the lattice in terms of alternating sums related to intersections. Again from our running example using the index sets $\{123,234,345\}$ we have:
\begin{eqnarray*}
H(12345) & = & H(123 \cup 234 \cup 345) \\
                  &  = & H(123) + H(234) + H(345) - H(23) - H(3) - H(34)  + H(3) \\
                 &  =  & H(123) + H(234) + H(345) - H(23) - H(34). 
 \end{eqnarray*}
Note that the complexity of the first inclusion-exclusion identity is simplified. The simplification here is an example of a property called the running intersection property, which we now explain.

\begin{Definition}
A sequence of index sets $I_1,\ldots, I_k$ has the {\em running intersection property} if for any
triple $I_i, I_j, I_k$ with $i < j < k$ we have
$I_i \cap I_k \subset I_j.$
\end{Definition} 
The result we have used is the following, which is known as Rota's inclusion-exclusion principle. See, e.g.,~\cite{klain1997introduction, knuth2006valuations}.
\begin{Lemma}
For a sequence of index sets $ \mathcal I = \{I_1, \ldots, I_k\}$ with the running intersection property and a valuation $\phi(I)$ on the distributed lattice generated by $\mathcal I$, we have the simplified inclusion-exclusion formula as:
$$ \phi\left( \cup_{i=1}^k I_i\right) = \sum_{i=1}^k \phi(I_i) - \sum_{i,j=1,\ i <j}^k \phi(I_i \cap I_j).$$
\end{Lemma} 

It is worth noting that the fact that the generators of the LCI have the running intersection property shows that if we only consider the first order intersections, then we have a decomposable graphical model \cite{lauritzen1988local}. 
We see this here in that the LCI gives many more inclusion-exclusion formulas than the single one above. For the LCI model every information associated with a margin, $E_I$ is the sum of the information associated with each conditional probability given by a $z_i$ (for $i \in I$).

\section{Discussion}

This paper is a contribution to the area of algebraic statistics but one which is somewhat different from the usual log-linear models' approach which leads to a toric representation for the raw probabilities. There {\em is} a toric representation, but for the marginal distributions. It is based on the simple observation that the distributive lattice definition of Hibi ideals is identically providing an exact mapping to a distributive lattice-based probabilistic model, called an LCI model. These were originally introduced for Gaussian models where the lattice arises from a particular selection of margins, but the formula is now quite general. The TDAG dual representation of the LCI model also has a purely algebraic representation in terms of a special bipartite graph arising from Alexander duality, a theory discovered independently of the probabilistic representation. The toric representation derives as the elimination ideal generated by certain products of variables $z_i$ that can be themselves associated with conditional independences, and the $\log z_{i}$ give the terms in (i) appropriate log-linear models (ii) time series models and (iii) models for Shannon information. We suggest that this algebraic underpinning of LCI models will help make them more widely applicable.

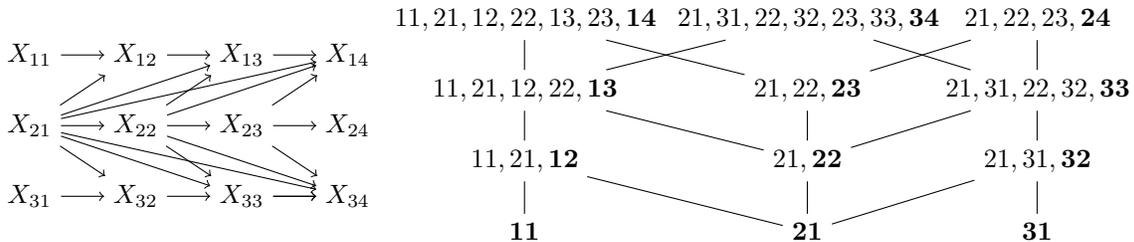
\begin{figure}
\begin{center}
\begin{tikzpicture} [scale = .47]

\node (11) at (4,8) {$X_{11}$};
\node (12) at (7,8) {$X_{12}$};
\node (13) at (10,8) {$X_{13}$};
\node (21) at (4,6) {$X_{21}$};
\node (22) at (7,6) {$X_{22}$};
\node (23) at (10,6) {$X_{23}$};
\node (31) at (4,4) {$X_{31}$};
\node (32) at (7,4) {$X_{32}$};
\node (33) at (10,4) {$X_{33}$};
\node (41) at (13,8) {$X_{14}$};
\node (42) at (13,6) {$X_{24}$};
\node (43) at (13,4) {$X_{34}$};
\draw[black][->](11)--(12);
\draw[black][->](12)--(13);
\draw[black][->](21)--(22);
\draw[black][->](22)--(23);
\draw[black][->](31)--(32);
\draw[black][->](32)--(33);
\draw[black][->](21)--(12);
\draw[black][->](21)--(13);
\draw[black][->](21)--(32);
\draw[black][->](21)--(33);
\draw[black][->](22)--(13);
\draw[black][->](22)--(33);

\draw[black][->](13)--(41);
\draw[black][->](23)--(42);
\draw[black][->](33)--(43);

\draw[black][->](22)--(41);
\draw[black][->](23)--(41);
\draw[black][->](21)--(41);

\draw[black][->](22)--(43);
\draw[black][->](23)--(43);
\draw[black][->](33)--(43);
\draw[black][->](21)--(43);


\node (11) at (18,3) {${\bf 11}$};
\node (21) at (26,3) {${\bf 21}$};
\node (31) at (32.5,3) {${\bf 31}$};
\node (12) at (18,5) {$11,21,{\bf 12}$};
\node (22) at (26,5) {$21,{\bf 22}$};
\node (32) at (32.5,5) {$21,31,{\bf 32}$};
\node (13) at (18,7) {$11,21,12,22,{\bf 13}$};
\node (23) at (26,7) {$21,22,{\bf 23}$};
\node (33) at (32.5,7) {$21,31,22,32,{\bf 33}$};
\node (1) at (18,9) {$11,21,12,22,13,23,{\bf 14}$};
\node (2) at (26,9) {$21,31,22,32,23,33,{\bf 34}$};
\node (3) at (32.5,9) {$21,22,23,{\bf 24}$};
\draw[black](11)--(12);
\draw[black](21)--(22);
\draw[black](31)--(32);
\draw[black](12)--(13);\draw[black](22)--(13);
\draw[black](22)--(23);
\draw[black](32)--(33);\draw[black](22)--(33);
\draw[black](21)--(32);\draw[black](21)--(12);

\draw[black](13)--(1);\draw[black](23)--(1);
\draw[black](13)--(2);\draw[black](33)--(2);
\draw[black](33)--(3);\draw[black](23)--(3);
\end{tikzpicture}
\caption{
(Left) Time series for $t=1,2,3,4$ while the curved edges 
$X_{ij}\rightarrow X_{ik}$ (for $i=1,2,3$ and $j< k+1$) are not depicted above. (right) The join-irreducible poset of the corresponding lattice. Note that $z_{ij}$ corresponds to the point which is the smallest subset containing $ij$ (marked in {\bf bold}.)
}\label{fig:lattice_time_series}
 \end{center}
\end{figure}

\medskip
\noindent
\textbf{Acknowledgments.} 
P.C. was partially supported by NSERC 05336-2019.
F.M. was partially supported by 
BOF/STA/201909/038, and FWO grants (G023721N, G0F5921N). E.S.C was supported by PID2020-116641GB-100 funded by MCIN/AEI/10.13039/501100011033.
H.W. was partially supported by project R-EDI-002 from Alan Turing Institute, London.


\bibliographystyle{abbrv.bst}
\bibliography{ref.bib}

\smallskip
\noindent
\footnotesize {\bf Authors' addresses:}

\bigskip 

\medskip

\noindent Department of Electrical and Computer Engineering, McGill University, Montreal, Canada
\\ E-mail address: {\tt peterc@cim.mcgill.ca}

\medskip

\noindent Department of Mathematics: Algebra and Geometry, Ghent University, 9000 Gent, Belgium \\
Department of Mathematics and Statistics,
UiT – The Arctic University of Norway, 9037 Troms\o, Norway
\\ E-mail address: {\tt fatemeh.mohammadi@ugent.be}

\medskip

\noindent Departamento de Matemáticas y Computación, Universidad de La Rioja, Logroño, Spain
\\ E-mail address: {\tt eduardo.saenz-de-cabezon@unirioja.es}

\medskip 

\noindent London School of Economics and Political Science, London WC2A 2AE, UK
\\ E-mail address: {\tt H.Wynn@lse.ac.uk}
\end{document}